\documentclass{amsart}

\newtheorem{theorem}{Theorem}
\numberwithin{theorem}{section}
\newtheorem{proposition}[theorem]{Proposition}
\newtheorem{conjecture}[theorem]{Conjecture}
\newtheorem{corollary}[theorem]{Corollary}
\newtheorem{definition}[theorem]{Definition}
\newtheorem{remark}[theorem]{Remark}
\numberwithin{equation}{section}
\newtheorem{lemma}[theorem]{Lemma}

\usepackage{amsfonts, amsmath, amsthm, amssymb, epsfig, bm, graphics,
  psfrag, latexsym, chngcntr, color, yfonts, mathtools, subfig, longtable,todonotes,enumitem,tikz-cd,appendix}
\usepackage[mathscr]{eucal}
\usepackage{hyperref}

\begin{document}
\title{Bigrading the symplectic Khovanov cohomology}
\author{Zhechi Cheng}
\maketitle 

\begin{abstract}
We construct a well-defined relative second grading on symplectic Khovanov cohomology from holomorphic disc counting. We show that it recovers the Jones grading of Khovanov homology up to an overall grading shift over any characteristic zero field, through proving that the isomorphism of Abouzaid-Smith can be refined as an isomorphism between bigraded cohomology theories. We prove it by constructing an exact triangle of symplectic Khovanov cohomology that behaves similarly to the unoriented skein exact triangle for Khovanov homology. We use a version of symplectic Khovanov cohomology defined for bridge diagrams and obtain an absolute homological grading in this construction. 

\end{abstract}

\section{Introduction}
In \cite{SS04}, Seidel and Smith defined a singly graded link invariant \textit{symplectic Khovanov cohomology} $Kh^*_{symp}(L)$. It is the Lagrangian intersection Floer cohomology of two Lagrangians in a symplectic manifold $\mathscr{Y}_n$. The manifold $\mathscr{Y}_n$ is built through taking a fiber of the restriction of the adjoint quotient map $\chi:\mathfrak{sl}_{2n}(\mathbb{C})\to Conf^0_{2n}(\mathbb{C})$ to a nilpotent slice $\mathcal{S}_n$.

A given link $L$ in $S^3$ can be realized as a braid closure of $\beta_L\in Br_{n}$ for some $n$ depending on $L$. $\beta_L\times id \in Br_{2n}$ gives a path in the configuration space $Conf^0_{2n}(\mathbb{C})$. The parallel transport induces a symplectomorphism of $\mathscr{Y}_n$ to itself, (precisely speaking, an arbitrarily large compact subspace of $\mathscr{Y}_n$). There is a distinguished Lagrangian submanifold $\mathcal{K}$ given by iterated vanishing cycles and let 
$(\beta_L\times id)(\mathcal{K})$ be its image under the parallel transport. $Kh^*_{symp}(L)$ is defined to be the Floer cohomology group \begin{equation}Kh^*_{symp}(L)=HF^{*+n+w}(\mathcal{K},(\beta_L\times id)(\mathcal{K}))\end{equation}, where $w$ is the writhe of $\beta_L$. There is a conjectural relation between $Kh^*_{symp}$ and the combinatorial Khovanov homology $Kh^{*,*}$:

\begin{conjecture}[Seidel-Smith, \cite{SS04}]\label{conj:SS} For any link $L\subset S^3$, $Kh^k_{symp}(L)\cong \bigoplus_{i-j=k}Kh^{i,j}(L)$. 
\end{conjecture}

This conjecture is true over any characteristic zero field, proved by Abouzaid and Smith in \cite{AS19}. For the cases of non-characteristic zero fields, only a few examples have been computed such that the theories are isomorphic, see \cite[Proposition 55]{SS04} for the case of trefoil
. 

Our results rely on the theorem of Abouzaid and Smith that Conjecture~\ref{conj:SS} is true over any characteristic zero field, so we assume the characteristic of the base field $\textbf{k}$ to be zero unless noted otherwise. 

It is also worth noting that we will be working with cohomology theories throughout the paper. Even if the Seidel-Smith invariant is called symplectic Khovanov homology in some contexts, it is of cohomological type, i.e. the differential raises the homological grading by $1$. Correspondingly, Khovanov homology is also a cohomology theory.



We will work with the framework of Manolescu's Hilbert scheme reformulation. The space $\mathscr{Y}_n$ can be embedded symplectically as an open subscheme into $Hilb^n(A_{2n-1})$, the $n$-th Hilbert scheme of the Milnor fibre of $A_{2n-1}$-surface singularity. 

One of the advantages of Manolescu's reformulation is that we can work with \textit{bridge diagrams} instead, which are decorated link diagrams obtained by breaking the link diagrams into $n$ pairwise disjoint $\alpha$-arcs and $n$ pairwise disjoint $\beta$-arcs such that $\beta$-arcs surpass $\alpha$-arcs at any intersection. These arcs give two Lagrangians $\mathcal{K}_\alpha$ and $\mathcal{K}_\beta$ in $\mathscr{Y}_n\subset Hilb^n(A_{2n-1})$. It is proved in \cite{M04} that $Kh^*_{symp}(L)=HF^{*+n+w}(\mathcal{K}_\alpha,\mathcal{K}_\beta)$ for a specific type of diagram, called a \textit{flattened braid diagram}, where $n$ is the number of strands and $w$ is the writhe of the corresponding braid. All braids can give rise to flattened braid diagrams but not all bridge diagrams are isotopic to flattened braid diagrams. 

Attempts were made to generalize symplectic Khovanov cohomology to arbitrary bridge diagrams, an $\mathbb{F}_2$-coefficients version by Hendricks-Lipshitz-Sarkar in \cite[Section 7]{HLS15} and a relatively graded version by Waldron in \cite[Section 6]{W09}. In this paper, we give an absolute grading to Waldron's construction as follow: 

\begin{theorem}\label{thm:bridge} For any oriented bridge diagram, let $w$ be the writhe of the diagram i.e. the number of positive minus the number of negative crossings, $rot$ be the rotation number of the diagram i.e. the number of counterclockwise minus the number of clockwise Seifert circles and $x_0$ be the generator whose coordinates are the starting point of each $\beta$-arcs. Then the Floer cohomology groups
\begin{equation} Kh^*_{symp}(L)=HF^{*+gr(x_0)+w+rot}(\mathcal{K}_\alpha,\mathcal{K}_\beta)
\end{equation}
are link invariants. 
\end{theorem} 

As Waldron proved in the relative case, see \cite[Theorem 1.1]{W09}, the absolutely graded invariant defined with bridge diagram is also canonical, i.e. we Floer groups from two equivalent bridge diagrams are related through a canonical isomorphism. 

The orientation of the diagram, especially for a link diagram, is crucial in computing the correction terms and locating $x_0$ (that there exactly are two generators whose coordinates are all endpoints and the choice depends on the orientation). Throughout this paper, we always assume that our bridge diagrams are oriented.


Abouzaid and Smith constructed an endomorphism $\phi$ of $CF^*(\mathcal{K}_\alpha,\mathcal{K}_\beta)$, which induces a generalized eigenspace decomposition of $HF^*(\mathcal{K}_\alpha,\mathcal{K}_\beta)$, see also \cite{AS16}. The eigenvalues give an additional grading on $HF^*(\mathcal{K}_\alpha,\mathcal{K}_\beta)$, called the \textit{weight grading}. 
We will only use a relative version of the weight grading because an absolute grading relies on choices of auxiliary data in symplectic geometry, called \textit{equivariant structures}, but we will not specify our choices of equivariant structures in this paper. We prove that for any bridge diagram, the relative weight grading recovers the Jones grading (or quantum grading in some contexts) of Khovanov cohomology

\begin{theorem}\label{thm:main} Symplectic Khovanov cohomology and Khovanov homology are isomorphic as bigraded vector spaces over any characteristic zero field, where the gradings are related by $k=i-j$ and $wt=-j+c$, where $k$ is the homological grading, $wt$ is the weight grading and $c$ is a correction term of the relative weight grading. \end{theorem}

By definition, the weight grading lives in $\bar{\textbf{k}}$, the algebraic closure of the base field. The theorem implies that the weight grading is integral. At the writing of this paper, the author does not know the correction term $c$ to define an absolute weight grading. Precisely speaking, we need to make a specific choice on the equivariant structures on the Lagrangians depending on the writhe, crossing number, and other properties of the bridge diagram. 

To prove Theorem~\ref{thm:main}, we show that Abouzaid-Smith long exact sequence of symplectic Khovanov cohomology groups, see \cite[Equation 7.9]{AS19} decomposes with respect to the weight grading. In other words, if we fix a weight grading $wt_1$ of the first group, the only non-trivial map can happen between a single weight grading $wt_2$ of the second group and $wt_3$ of the third group. 

\begin{equation}\ldots\to HF^{*,wt_1}(L_+)\to HF^{*,wt_2}(L_0)\to HF^{*+2,wt_3}(L_\infty)\to HF^{*+1,wt_1}(L_+)\to \ldots \end{equation}

In the singly-graded case, the isomorphism of Abouzaid and Smith between symplectic Khovanov cohomology and Khovanov homology can be used in building a commutative diagram between the exact triangle above and the exact triangle of unoriented skein relation in Khovanov homology. We show the maps between the exact triangles given by the isomorphism of Abouzaid-Smith are bigraded by induction on the number of crossings. Abouzaid-Smith purity result \cite[Theorem 1.1]{AS16} leads to a computation for crossingless diagram of an unlink:

\begin{proposition}[Abouzaid-Smith, \cite{AS16}] If $L\subset S^3$ is an unlink represented by a crossingless diagram, there exists a choice of equivariant structures on Lagrangians such that for any element $x\in Kh^k_{symp}(L)$, $wt(x)=k$. 
\end{proposition} 
This finishes the proof of Theorem~\ref{thm:main}. Our argument so far is diagrammatic, we have not proved that the relative weight grading is independent of bridge diagrams yet. Now that we know the relative weight grading recovers the Jones grading up to an overall grading shift for any diagram and the fact that Jones grading is independent of link diagrams, we prove a conjecture of Abouzaid-Smith, 

\begin{theorem} \label{thm:relative} The relative weight grading on $Kh^*_{symp}(L)$ is independent of the choice of link diagram. 
\end{theorem}

It is worth noting that the proof of Theorem~\ref{thm:relative} is not internal to symplectic geometry, and the invariance of the relative weight grading relies on the well-definition of the Jones grading in combinatorial Khovanov homology.

\textbf{Organization:} The paper is organized as follows: In Section~\ref{ch2}, we review the definition of symplectic Khovanov cohomology and construct an absolute grading on symplectic Khovanov cohomology of bridge diagrams. In Section~\ref{ch3}, we give a precise definition of the weight grading and construct a bigraded unoriented skein exact triangle of symplectic Khovanov cohomology. In Section~\ref{ch5}, we prove the main theorem by showing Abouzaid-Smith's isomorphism between symplectic Khovanov cohomology and combinatorial Khovanov homology preserves the second grading. 

\textbf{Acknowledgement:} The author would like to thank his thesis advisor Mohammed Abouzaid for his guidance throughout this project. He also thanks Kristen Hendricks, Mikhail Khovanov, Francesco Lin, Robert Lipshitz, and Ivan Smith for many helpful discussions on the subject. Finally, he thanks Columbia University for its supportive studying environment during the early stage of writing the paper and Institut Mittag-Leffler for its hospitality during the later stage of preparation of this paper. The author was partially supported by his thesis advisor's NSF grants DMS-1609148, and DMS-1564172, and by Institut Mittag-Leffler Junior Fellowship.

\section{A review of symplectic Khovanov cohomology $Kh^{*}_{symp}(L)$}\label{ch2}
We will briefly review the original definition of symplectic Khovanov cohomology and give a formal definition of symplectic Khovanov cohomology of a bridge diagram in Section~\ref{3.2}. We will discuss the homological grading in Section~\ref{3.3}. The construction of this section is not restricted to characteristic zero fields.

\subsection{Symplectic Khovanov cohomology for bridge diagrams}\label{3.2}

In \cite{SS04}, the link invariant $Kh^*_{symp}(L)$ is first introduced by Seidel and Smith as the Lagrangian intersection Floer cohomology of two Lagrangians in $\mathscr{Y}_n$, constructed as a nilpotent slice in $\mathfrak{sl}_{2n}(\mathbb{C})$. But it is defined through a link diagram presented as a braid closure and the Floer cohomology is computed through the braid action on $\mathscr{Y}_n$. Manolescu introduced a reformulation using Hilbert schemes in \cite{M04}, which is easier to visualize and more similar to other low dimensional invariants, such as Heegaard Floer homology. In this subsection, Floer homology groups are relatively graded. 

Let us start with Hilbert schemes of points on surfaces. We choose our algebraically closed field to be $\textbf{k}=\mathbb{C}$ and $X$ to be a complex variety. The Hilbert scheme of $n$ points on $X$, $Hilb^n(X)$ is defined to be closed 0-dimensional subschemes of $X$ of length $n$. An important part of this variety is a subvariety consisting of $n$ distinct points but its diagonal where points collide together is really complicated. However, we have the following \textit{Hilbert-Chow morphism} from \cite{N99}:
\begin{proposition} The Hilbert-Chow morphism $\pi$ is a natural morphism from the Hilbert scheme of $n$ points on $X$ to the $n$-fold symmetric product of $X$ such that
\begin{equation} \pi(Z)=\sum_{x\in X} length(Z_x)[x]
\end{equation}
Moreover, if $X$ is complex 1-dimensional, then $\pi$ is an isomorphism. If $X$ is complex 2-dimensional, then $\pi$ is a resolution of singularities and $Hilb^n(X)$ is smooth. 
\end{proposition}


Now we specify our complex surface. Consider the following complex surface \begin{equation}S=\{(u,v,z)\in \mathbb{C}^3|u^2+v^2+p(z)=0\}\in \mathbb{C}^3\end{equation}, where $p(z)=(z-p_1)\ldots(z-p_{2n})$. This is isomorphic to a fibre of nilpotent slice of block size $(1,2n-1)$ with eigenvalues $p_i$ in Seidel-Smith construction. Seidel-Smith invariant is defined in the nilpotent slice $\mathscr{Y}_n$ of block size $(n,n)$, which is proved to be an open subscheme of $Hilb^n(S)$, by Manolescu. Denote its complement 
\begin{equation} D_r=Hilb^n(S)\backslash \mathscr{Y}_n
\end{equation}
which is a complex co-dimension 1 subvariety and, in fact, a relative Hilbert subscheme given by all elements with length less than $n$. 

To fully characterize $\mathscr{Y}_n$, consider a projection $i:S\to\mathbb{C}$ such that $i(u,v,z)=z$, then 
\begin{equation}\mathscr{Y}_n=\{I\in Hilb^n(S)|i(I)\text{ has length }n\}
\end{equation} Manolescu proved in \cite[Proposition 2.7]{M04} that $\mathscr{Y}_n$ is biholomorphic to the space $\mathscr{Y}_{n,\tau}$ obtained by a fibre of nilpotent slice by Seidel-Smith in \cite{SS04}. 

We need to construct Lagrangians from a link diagram. A \textit{bridge diagram} $D$ for a link $L$ is a triple $(\vec{\alpha},\vec{\beta},\vec{p})$, where $\vec{p}=(p_1, p_2,\ldots, p_{2n})$ are $2n$ distinct points in $\mathbb{R}^2$, $\vec{\alpha}=(\alpha_1, \alpha_2,\ldots, \alpha_n)$ are $n$ pairwise disjoint embedded arcs and $\vec{\beta}=(\beta_1, \beta_2,\ldots, \beta_n)$ are also pairwise disjoint embedded arcs such that $\partial(\cup \alpha_i)=\partial(\cup \beta_i)=\{p_1,\ldots,p_{2n}\}$ and if we let the $\beta$ arcs surpass the $\alpha$ arcs at the intersections in $\mathbb{R}^3$, we get $L$. 

For each arc $\alpha_i$ or $\beta_i$, we can associate a Lagrangian sphere $\Sigma_{\alpha_i}$ or $\Sigma_{\beta_i}$ in S through the following equation: 
\begin{equation}\Sigma_{\alpha_i}=\{(u,v,z)\in S| z\in \alpha_i, u, v \in\sqrt{-p(z)}\mathbb{R}\}\end{equation} \begin{equation}\Sigma_{\beta_i}=\{(u,v,z)\in S| z\in \beta_i, u, v \in\sqrt{-p(z)}\mathbb{R}\}\end{equation} 

For each interior point of the arc, we have an $S^1$, while each end point gives a point. So $\Sigma_{\alpha_i}$ and $\Sigma_{\beta_i}$ are copies of $S^2$ and it is easy to see that they are Lagrangians for an appropriate choice of K\"ahler form, see \cite[Section 4]{M04}. 

These spheres enable us to build two Lagrangians in $Hilb^n(S)$ and, in fact, in $\mathscr{Y}_n$ by 
\begin{equation}\mathcal{K}_\alpha=\Sigma_{\alpha_1}\times\Sigma_{\alpha_2}\times\ldots\times\Sigma_{\alpha_n}
\end{equation} 
\begin{equation}\mathcal{K}_\beta=\Sigma_{\beta_1}\times\Sigma_{\beta_2}\times\ldots\times\Sigma_{\beta_n}
\end{equation}

It is worth pointing out that $\mathcal{K}_\alpha$ and $\mathcal{K}_\beta$ do not intersect transversely. The intersections of these spheres are points in the fibers of the endpoints of the $\alpha$ arcs and $\beta$ arcs and circles $S^1$ in the fibers of interior intersections of those arcs. Thus, we could have some tori as intersections. To deal with this, we should perturb one of the Lagrangians, see \cite[section 6.1]{M04}, or use Floer theory with clean intersections, like \cite{P99}

Because $\mathcal{K}_\beta$ is a product $\Sigma_{\beta_1}\times\Sigma_{\beta_2}\times\ldots\times\Sigma_{\beta_n}$, it will be easier to perturb each $\Sigma_\beta$ in the complex surface $S$. Let $N=\Sigma_\alpha\cap\Sigma_\beta$ be the portion of the intersection consisting of copies of $S^1$ intersections in $S$ and $V$ be a neighborhood of $N$. Then according to Weinstein \cite{W73}, we use the standard height function on $N$ as a Morse-Smale function that is required by the context. Then $\Sigma_\beta$ can be isotopied into $\Sigma_\beta'$ such that they are identical outside of V and intersects $\Sigma_\beta\cap V$ exactly at the maximum and the minimum of our height function. The resulting Floer cochain $CF^*(\Sigma_\alpha,\Sigma_\beta')$ will be quasi-isomorphic to $CF^*(\Sigma_\alpha,\Sigma_\beta)$. 

From now on, we always treat $\mathcal{K}_\beta$ as the original Lagrangian perturbed whenever we consider Lagrangian intersections, so it will intersect transversely with $\mathcal{K}_\alpha$ at isolated points. So each intersection at the interior of arcs $\alpha_i$ and $\beta_j$ now gives the intersection of $\Sigma_{\alpha_i}$ and $\Sigma_{\beta_j}$ at two points instead of a circle. 


\begin{proposition}\label{prop:rela} $($\cite[Theorem 1.2]{M04}, \cite[Theorem 1.1 and Theorem 4.12]{W09}$)$ For any bridge diagram $D$, The Floer cohomology $HF^*(\mathcal{K}_\alpha,\mathcal{K}_\beta)$ in $\mathscr{Y}_n=Hilb^n(S)\backslash D_r$ is canonically isomorphic to Seidel-Smith symplectic Khovanov homology $Kh^*_{symp}(L_D)$, where $L_D$ is the link represented by the bridge diagram. 
\end{proposition}

With Proposition~\ref{prop:rela} in mind, we finally define:  
\begin{definition}\label{def:symp} The symplectic Khovanov cohomology $Kh^*_{symp}(D)$ of a bridge diagram $D$ is defined to be $HF^*(\mathcal{K}_\alpha,\mathcal{K}_\beta)$ in $\mathscr{Y}_n=Hilb^n(S)\backslash D_r$. 
\end{definition}

\begin{remark} We have not yet given an absolute grading for this definition yet, and in fact, Proposition~\ref{prop:rela} is proved in the relatively graded case. Manolescu has an absolute grading in \cite{M04} but it only works with flattened braid diagrams, where explicit choices can be made on Lagrangians to construct an absolute Maslov grading. Waldron's construction in \cite{W09} works for any bridge diagram and his isomorphism is canonical but only works in the relatively graded case. 
\end{remark}

In the rest of the paper, whenever we mention $Kh^*_{symp}(L)$, we will always be working with Definition~\ref{def:symp}, unless noted otherwise.

\subsection{An absolute grading for bridge diagrams}\label{3.3}
In this subsection, we prove Theorem~\ref{thm:bridge} by showing that the homological grading shifted by $gr(x_0)+rot+w$ is invariant under isotopy, handleslide and stabilization. We reiterate our conventions here that are crucial in defining the absolute grading. Our link is oriented. At each intersection, the $\beta$-arc surpasses the $\alpha$-arc. The distinguished generator has all the coordinates at the starting points of every $\beta$-arc. 

The idea of this correction term is from Droz-Wagner in \cite{DW09}, where they worked with grid diagrams obtained from deforming flattened braid diagrams. A flattened braid diagram of a braid $b\in{Br_n}$ is a special bridge diagram associated to the braid $b$ whose marked points $\mu_i$ are placed on the real line with an order by their indices, $\alpha$-arcs are segments on the real line connecting $\mu_{2i-1}$ and $\mu_{2i}$, $\beta$-arcs are the images of $\alpha$-arcs after applying the braid $b\times id\in Br_{2n}$ action on the plane with $2n$ punctures at the marked points. 

\begin{remark}An interesting example is that if we take a bridge diagram that represents a braid closure, the rotation number is exactly the number of strands (because all Seifert circles are going counterclockwise,) and the writhe of the diagram matches the writhe of the braid. The distinguished generator has homological grading $0$ in any flattened braid diagram. Thus the correction term $gr(x_0)+rot+w$ for the braid closure diagram agrees with Manolescu's correction term $n+w$ of the flattened braid diagram, where these two diagrams are isotopic as link diagrams.
\end{remark}

Geometrically, such an absolute grading requires us to grade our Lagrangians by choosing some sections of the canonical bundle. In the braid closure setup, two Lagrangians are related by some fibered Dehn twists and thus the choice on the second Lagrangian can be induced from the first one canonically. In the general bridge diagram case (not a flattened braid diagram), there is no easy way to assign a choice to the second Lagrangian, so instead, we assign the generator $x_0$ to have the grading $0$. 

Manolescu pointed out a combinatorial method to compute the relative homological grading in \cite[Subsection 6.2]{M04}. In short, we replace each $\alpha$-arc $\alpha_i$ with an oriented (the orientation does not matter) figure eight $\gamma_i$ in a small neighborhood. We arrange our diagram such that each $\beta$-arc is horizontal and each figure eight is vertical wherever they intersect. Each intersection of Lagrangian spheres $\Sigma_{\alpha_i}$ and $\Sigma_{\beta_j}$ corresponds to an intersection of the arc $\gamma_i$ and the figure eight $\beta_j$. We travel along with the figure eight $\gamma_i$ and mark the points that have horizontal tangent lines. We assign $+1$ to those marked points if it is locally oriented counterclockwise and $-1$ if clockwise. We start moving along the oriented figure eight, at the starting point of the $\beta$-arc with $0$ (that it is essentially still relative but we are going to remove the grading of this generator in the correction term anyway) and change this number by the number on the marked point whenever we reach one. In this way, each intersection of a $\beta$-arc and the figure eight $\gamma_i$ will be labeled with a number, see Figure~\ref{fig:Isotopy1} for an idea of the computation. Each of the generators of the Floer cohomology group corresponds to $n$ intersection points and the sum of the labeled numbers will be its grading. 

\begin{proposition}$HF^{*+gr(x_0)+rot+w}(\mathcal{K}_\alpha,\mathcal{K}_\beta)$ is invariant under bridge diagram isotopy. 
\end{proposition}
\begin{proof} An isotopy of a bridge diagram induces Hamiltonian isotopic Lagrangians and thus keeps the relatively graded group $HF^{*}(\mathcal{K}_\alpha,\mathcal{K}_\beta)$ unchanged. 

If no crossing is introduced or removed, it is easy to see that all three components in the correction term remain the same. If crossings are introduced, it must come from one of the following cases, as shown in Figure~\ref{fig:Isotopy}: 
\begin{figure}
\includegraphics[width=6cm]{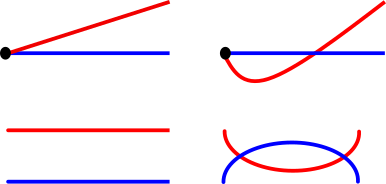}
\centering
\caption{Two possible bridge diagram isotopies corresponding to Reidemeister I and II moves. }
\label{fig:Isotopy}
\end{figure}

These two isotopies correspond to Reidemeister move I and II, (whereas Reidemeister move III is from a combination of (de-)stabilizations and handleslides.)

In the first case, there are two subcases--the $\beta$-arc oriented to the left and to the right. The easier subcase is the left-going $\beta$-curve. The crossing is a negative crossing and the bigon region gives an additional counterclockwise Seifert circle. The distinguished generator $x_0$ has no coordinate in this region and thus its grading does not change. The total change of $gr(x_0)+rot+w$ is $0$. 

\begin{figure}
\includegraphics[width=6cm]{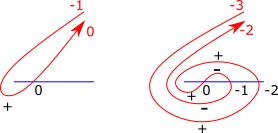}
\centering
\caption{Isotopy corresponding to Reidemeister I move with intersections marked with gradings and horizontal tangent points marked with signs.}
\label{fig:Isotopy1}
\end{figure}

The other subcase is more subtle. The crossing is still negative but the Seifert circle is now clockwise, which changes the correction term by $-2$. The distinguished generator $x_0$ indeed takes a coordinate at the black dot. With the help of figure eights in Figure~\ref{fig:Isotopy1}, we can visualize that the relative grading is changed by $2$, which cancels the contribution of the other two terms. In fact, we can explicitly describe the change of the cochain. The bigon region in the complex plane lifts to some holomorphic discs in the complex surface $S$, connecting one of the two generators at the interior intersection to the endpoint. One of the moduli spaces has indeed dimension $1$ that contributes to differentials cancelling pair of generators with identical coordinates except one coordinate in the picture. At chain level, the isotopy created three copies of the complexes that have a coordinate at the black dot, but two of them are canceled on the homology level. The surviving copy should have the same grading as the original chain. It is not hard to see the other moduli space of the bigon region has dimension $2$, and thus locally changing coordinate from the lower grading coordinate at the intersection to the black dot increases the grading by $2$. So in the new diagram, the grading of the distinguished generator increases by $2$ relatively to other generators, which cancels the contribution by the sum of writhe and rotation number. 

In the second case which corresponds to Reidemeister II move, regardless of the orientation, two new crossings always have different signs so the writhe remains unchanged. If the two components are oriented in the same direction, the Seifert surface remains unchanged and thus the rotation number is also the same. If the two components are oriented differently, there will be two subcases, whether the two strands are from the same component or not. In the first subcase, locally one Seifert circle breaks into three circles, but the middle one is of a different direction. In the second subcase, there are two Seifert circles of the same direction in each picture. In either case, the rotation number remains the same as well. Apart from the two new intersections in the new diagram, all other generators are from the diagram before the isotopy and it is not hard to see that the gradings of such generators remain unchanged. Thus $gr(x_0)$ remains unchanged as well relative to other generators. 
\begin{figure}
\includegraphics[width=6cm]{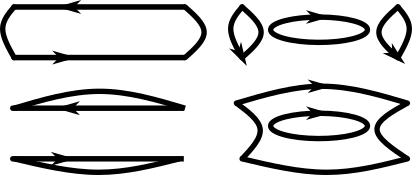}
\centering
\caption{Two cases for Reidemeister II isotopy when two components are oriented differently.}
\label{fig:Isotopy2}
\end{figure}
\end{proof}

\begin{proposition}$HF^{*+gr(x_0)+rot+w}(\mathcal{K}_\alpha,\mathcal{K}_\beta)$ is invariant under stabilization. 
\end{proposition}
\begin{proof} Stabilizations will not change the writhe or the rotation number of the diagram. It suffices to show that the grading of the homology remains relatively the same to the distinguished generator $x_0$. It is true because Waldron's proof of invariance is based on chain level that the isomorphism is in fact on chain level so $gr(x_0)$ is not changed relative to cohomology, see \cite[Lemma 5.19]{W09}.
\end{proof}

The handleslide invariance is the hardest among the three. Before the proof, we make some topological observations. With the isotopy invariance, we should arrange our diagram to be as simple as possible. 

We explicitly describe the process of the handleslide of arc $\alpha_2$ over $\alpha_1$. First, we choose the boundary circle of slightly flattened $\alpha_1$, such that each midpoint intersection of any $\beta$ arc locally intersect the circle exactly twice, and the $\beta$ arcs connecting the endpoints of $\alpha_1$ locally intersect the circle exactly once. Then pick any path $\gamma$ from $\alpha_2$ to the circle that does not intersect any $\alpha$ arc in its interior, and perform a connect sum such that each intersection between the path and the $\beta$ arc creates exactly two intersections on the connected sum. 

\begin{proposition}$HF^{*+gr(x_0)+rot+w}(\mathcal{K}_\alpha,\mathcal{K}_\beta)$ is invariant under handleslide. 
\end{proposition}
\begin{proof} It is clear that new intersections are created in pairs, and each pair contains exactly one negative and one positive crossing. Thus the writhe of the diagram remains unchanged. 

Next, we show that the rotation number remains unchanged as well. This breaks into two parts. 

First, we study the intersections created around the path $\gamma$. Locally, there are a series of parallel $\beta$-arcs intersecting two $\alpha$-arcs. The two $\alpha$-arcs are parallel to the path $\gamma$ and always oriented in different directions. 

Let us call a $\beta$-arc with positive sign if it goes from left to right, and negative otherwise. We study the local picture with only two adjacent $\beta$-arcs. There are four cases in total shown in Figure~\ref{fig:HS1} In the case $++$, the first Seifert circle remains the same. In the case $--$, the second circle remains the same. The other Seifert circle is connected to another section. 
\begin{figure}
\includegraphics[width=6cm]{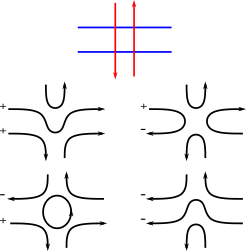}
\centering
\caption{All four possible Seifert circles at the intersections near the path $\gamma$, depending on the orientation of $\beta$-arcs.}
\label{fig:HS1}
\end{figure}

In the case $+-$, there are two subcases, whether the two $\beta$-arcs belong to the same strand or different strands, similar to the Reidemeister II move argument. It will either break a clockwise circle into two, or merging two counterclockwise circles into one. In either case, the rotation number is reduced by $1$. 

In the case $-+$, one counterclockwise circle is created. The other two Seifert circles are studied already in other cases. Thus the rotation number is increased by $1$. 

To sum this up, each time the sign change from $+$ to $-$, the rotation number decreases by $1$, and each time the sign change from $-$ to $+$, the rotation number increases by $1$. At the intersection of the path $\gamma$ and $\alpha_1$, $\alpha_1$ is oriented to the right (to match the orientation of the $\alpha$-arcs in the handleslide picture,) and the circle will remain the same if the first $\beta$ is positive, and decreases the rotation number by $1$ if the first $\beta$ is negative. Thus, if the last $\beta$-arc is positive, the rotation number remains the same, if the last one is negative, the rotation number is reduced by $1$. The last $\beta$-arc is also related to the next part of the argument. 

The other part we need to compute is the contribution of new crossings created at the circle near $\alpha_2$. Without loss of generality, let us assume $\alpha_2$ is oriented from left to right. The part in the dashed box of Figure~\ref{fig:HS2} is exactly the same as the picture before the handleslide except for the last endpoint intersection, (and, in fact, the last intersection correspond to the intersection on the top-right if we keep going from the $\alpha$-arc.) The rest of the region is two parallel $\alpha$-arcs oriented differently, intersecting a series of parallel $\beta$-arcs. The study of these new intersections is similar to the first half of the argument and the conclusion is similar, except now the sign depends on the first $\beta$-arc. We claim that the rotation number remains the same if the first crossing is positive, and increases by $1$ if the first crossing is negative. 

\begin{figure}
\includegraphics[width=8cm]{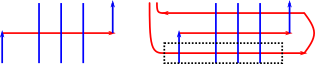}
\centering
\caption{All four possible Seifert circles at the intersections near the path $\gamma$, depending on the orientation of $\beta$-arcs.}
\label{fig:HS2}
\end{figure}

Now we only need to check the last $\beta$ of the first part and the first $\beta$ of the last part. If both are of the same sign, the Seifert circle does not change. If the $\beta$ from the first half if $+$ and the $\beta$ from the second half is $-$, then it created a counterclockwise Seifert circle, rotation number increases by $1$. In the symmetric case, the rotation number decreases by $1$. In any of the cases above, the total change of rotation number if $0$. 

Lastly, we need to check the relative grading of the homology remains the same relative to the distinguished generator $x_0$. First of all, the isomorphism of two homology groups is induced by a continuation map on the cochain level. It is also easy to see that there is an injection from the generators of the original cochain to the ones of the cochain after the handleslide. It is easy to see that the homological gradings of the corresponding generators are not changed from the definition of figure-eights, because we only changed one of the arcs and we essentially only applied an isotopy to one of the $\alpha$ arcs if we allow the isotopy to move across other marked points and $\alpha$ arcs. Moreover, the distinguished generator is obviously one of the generators that are preserved in the correspondence, and thus the homological grading relative to the distinguished generator $x_0$ remains the same. 


\end{proof}

Combining the propositions in this subsection, we complete a proof for Theorem~\ref{thm:bridge} that we obtain an absolutely graded version of symplectic Khovanov cohomology for any bridge diagram. 

\section{Weight grading on $Kh^*_{symp}(L)$}\label{ch3}

\subsection{Construction of weight grading}\label{3.4}
In this subsection, we define the weight grading $wt$ on $Kh^*_{symp}$, following the idea of Abouzaid-Smith, see \cite[Section 3]{AS16} for more details. The idea of constructing such a grading is building an automorphism, more precisely, the linear term of a \textit{non-commutative vector field}, or a \textit{nc vector field} \begin{equation}\phi:CF^*(\mathcal{K}_\alpha,\mathcal{K}_\beta)\to CF^*(\mathcal{K}_\alpha,\mathcal{K}_\beta)\end{equation} which preserves the homological grading and commutes with the differential. Then it will induce an automorphism $\Phi$ on cohomology. If $x$ is an eigenvector of eigenvalue $\lambda$, then we define $wt(x)=\lambda$. Most results of this section work for any characteristic field but we will restrict our discussion to characteristic zero ones, see Remark~\ref{re:field} for more detail. 

The idea of defining the map $\phi$ is to study certain moduli spaces of holomorphic maps in a partially compactified space $\bar{\mathscr{Y}}_n$ of $\mathscr{Y}_n$:
\begin{itemize} 
\item Let $Z=\bar{\bar{A}}_{2n-1}$ be the blow up of $\mathbb{P}^1\times\mathbb{P}^1$ at $2n$ points. It admits a Lefschetz fibration structure over $\mathbb{P}^1$ and let $F_\infty$ be its fiber at $\infty$ and $s_0$, $s_\infty$ be two sections. The original surface is simply \begin{equation}A_{2n-1}=\bar{\bar{A}}_{2n-1}\backslash(F_\infty\cup s_0\cup s_\infty)\end{equation}
\item Let the intermediate space be \begin{equation}\bar{A}_{2n-1}=\bar{\bar{A}}_{2n-1}\backslash(F_\infty)\end{equation}.
\item Through taking Hilbert scheme, the projective variety is $\bar{\bar{M}}=Hilb^n(Z)$.
\item $D_0$ is the divisor of subschemes whose support meets $s_0\cup s_\infty$.
\item $D_\infty$ is the divisor of subschemes whose support meets $F_\infty$. 
\item $D_r$ is the relative Hilbert scheme of the projection $Z\to \mathbb{P}^1$, which is a divisor supported on a compactification of the complement of $\mathscr{Y}_n$. 
\item $\bar{\mathscr{Y}}_n=\bar{\bar{M}}\backslash D_\infty=Hilb^n(\bar{A}_{2n-1})$. 
\end{itemize}

The moduli space was originally used by Seidel and Solomon in \cite{SS12} for the q-intersection number and later used by Abouzaid and Smith in \cite{AS16} for the weight grading. Let $\mathscr{R}^{k+1}_{(0,1)}$ be the moduli space of holomorphic classes of closed unit discs with the following additional data:
\begin{itemize} 
\item two marked points $z_0=0$ and $z_1\in (0,1)$. 
\item $k+1$ boundary punctures at $p_0=1$ and $k$ others $p_1,\ldots p_k$ placed counterclockwise. 
\end{itemize}

We define \begin{equation}\mathscr{R}^{k+1}_{(0,1)}(x_0;x_k,\ldots,x_1)\end{equation} to be the moduli space of finite energy holomorphic maps $u:\mathbb{D}\to \bar{\mathscr{Y}}_n$ such that 
\begin{itemize}
\item $u^{-1}(D_r)=\emptyset$
\item $u^{-1}(D_0)=z_0$
\item $u^{-1}(D'_0)=z_1$, where $D'_0$ is a divisor linearly equivalent to $D_0$ but shares no irreducible component with $D_0$. 
\item $u(p_i)=x_i$
\end{itemize}

Now we are ready to introduce our automorphism $\Phi$. In the case of $k=1$, the virtual dimension of $\mathscr{R}^{2}_{(0,1)}(x,y)$ is the difference of the homological grading $gr(x)-gr(y)$, see \cite[Lemma 3.16]{AS16}. It makes sense to consider a map of degree $0$ that counts holomorphic discs in $\mathscr{R}^{2}_{(0,1)}(x,y)$: 
\begin{equation} b^1(x)=\sum_{y|gr(y)=gr(x)}\#\mathscr{R}^{2}_{(0,1)}(x,y)y
\end{equation}

\begin{remark} We stated the definition above without orientation. But in fact, the moduli spaces $\mathscr{R}^{k+1}_{(0,1)}(x_0;x_k,\ldots,x_1)$ admit natural orientations relative to the orientation lines $o_{x_i}$ and the moduli spaces of curves $\mathscr{R}^{k+1}_{(0,1)}$, by fixing orientation for $\mathscr{R}^{k+1}_{(0,1)}$ by identifying the interior points of this moduli space with an open subset of $(0,1)\times (\partial\Delta)^{k+1}$ and twist the overall orientation by $\sum_{i=1}^k i gr(x_i)$. 
\end{remark}

$\tilde{b}^1$ is the linear part of a Hochschild cochain $\tilde{b}\in CC^*(\mathcal{F}(M),\mathcal{F}(M))$ if we allow multiple inputs instead of one single input $y$. But this cochain is not closed. Two additional string maps 
\begin{equation}CO(gw_1)\text{, where $gw_1$ bounds the Gromov-Witten invariant $GW_1$, }\end{equation}  
\begin{equation} co(\beta_0)\text{, where $\beta_0$ bounds the intersection of $D_0$ and $D_0'$,}\end{equation} are needed to make a closed cochain and moreover a $nc$ vector field. By adding these two terms, we obtain a Hochschild cocycle $b$ and its linear part $b^1$. 




By definition, given an (exact) Lagrangian $L$, the obstruction to the existence of an equivariant structure $c$ is given by the first term of the $nc$ vector field $b^0|_L\in HF^1(L,L)\cong H^1(L)$, and the set of choices when this vanishes is an affine space equivalent to $H^0(L)$. In our case, the Lagrangians are a product of $S^2$, thus $H^1(L)\cong\{0\}$ and $H^0(L)\cong \textbf{k}$. 

Since we do not assume that $b^0$ vanishes, for a cocycle $b\in CC^1(\mathcal{F}(M),\mathcal{F}(M))$ and equivariant objects $(\mathcal{K}_\alpha,c_\alpha)$ and $(\mathcal{K}_\beta,c_\beta)$, the linear term $b^1$ is not always a chain map. However, we can define a chain map \begin{equation} \phi(x)=b^1(x)-\mu^2(c_\alpha,x)+\mu^2(x,c_\beta) \end{equation}. It induces an endomorphism $\Phi$ on $HF^*(\mathcal{K}_\alpha,\mathcal{K}_\beta)$. If we consider the (generalized) eigenspace decomposition, the eigenvalue of the generalized eigenvector $x$ will be its weight grading denoted as $wt(x)$. 

\begin{remark}\label{re:field} The construction above works in any field $\mathbf{k}$. The weight grading is a priori indexed by elements of the algebraic closure $\bar{\mathbf{k}}$. Working with characteristic zero field will not only enable us to use the result of Abouzaid-Smith that identifies symplectic Khovanov cohomology and Khovanov homology, but also will make it possible to make the weight grading integral, in contrast to finite fields. If the weight grading lives in the algebraic closure of a finite field, it is impossible to make it integral using conventional methods. 
\end{remark}

We learn from \cite[Lemma 2.12]{AS16} that since $HF^0(\mathcal{K},\mathcal{K})\cong\textbf{k}$, changing equivariant structures shifts the overall weight by a constant \begin{equation} \Phi_{(\mathcal{K}_\alpha,c_\alpha),(\mathcal{K}_\beta,c_\beta)}= \Phi_{(\mathcal{K}_\alpha,c_\alpha+s_\alpha),(\mathcal{K}_\beta,c_\beta+s_\beta)}+(s_\alpha-s_\beta)id \end{equation}

So we have the following definition of the relative weight grading:
\begin{definition} Let $\Phi$ be the endomorphism constructed above on $HF^*(\mathcal{K}_\alpha,\mathcal{K}_\beta)$ and $x$ be an eigenvector of $\Phi$. The relative weight grading $wt(x)$ is defined to be the eigenvalue of $x$. This construction relies on auxiliary choices of equivariant structures on $\mathcal{K}_\alpha$ and $\mathcal{K}_\beta$, but different choices of such structures will only change all gradings by a fixed number and thus $wt(x)$ as a relative grading is independent of choices of equivariant structures.
\end{definition}

\begin{remark} With given equivariant structures on $\mathcal{K}_\alpha$ and $\mathcal{K}_\beta$, we can compute an absolute weight grading. But at the time of writing this paper, the author does not know the choices that would give a correction term independent of the link diagram. 
\end{remark}

Then we can rephrase Abouzaid-Smith's purity result as follow: 
\begin{proposition}\cite[Proposition 6.11]{AS16}\label{prop:nocrossing} Let $D$ be a bridge diagram without any crossings. Then we can choose $c_\alpha$ on $\mathcal{K}_\alpha$ and $c_\beta$ on $\mathcal{K}_\beta$ such that for any $x\in HF^*(\mathcal{K}_\alpha,\mathcal{K}_\beta)$, we have $wt(x)=gr(x)$. \end{proposition} 

If we recall the Khovanov homology of an unlink of $k$-component $U_k$ is
\begin{equation} Kh^{*,*}(U_k)=\bigotimes^k(\textbf{k}_{(0,1)}\oplus\textbf{k}_{(0,-1)}).
\end{equation}
With the choice of the equivariant structures by Abouzaid and Smith, we have
\begin{equation} Kh^{*,*}_{symp}(U_k)=\bigotimes^k(\textbf{k}_{(1,1)}\oplus\textbf{k}_{(-1,-1)})
\end{equation}
This proposition is a special case of our main theorem with crossing number equals to 0, if we relate gradings $(gr,wt)$ on symplectic Khovanov cohomology and $(i,j)$ on Khovanov homology with the formula:
\begin{equation} i=gr-wt \end{equation}
\begin{equation} j=-wt \end{equation}

\subsection{Floer product and weight grading}\label{2.4}
In \cite[Section 3]{AS16}, Abouzaid and Smith pointed out an important fact that the weight grading is compatible with Floer products, without actually phrasing and proving the precise statement. Also in \cite[Equation 4.9]{SS12}, Seidel and Solomon discussed the derivation property in a similar setup. We prove the following proposition: 

\begin{proposition}\label{prop:product} Let $\mathcal{K}_0$, $\mathcal{K}_1$, $\mathcal{K}_2$ be compact Lagrangians given by crossingless matchings in $\mathscr{Y}_n$. For any eigenvector $\alpha\in HF^*(\mathcal{K}_1,\mathcal{K}_2)$ and $\beta\in HF^*(\mathcal{K}_0,\mathcal{K}_1)$, we have \begin{equation}wt(\mu^2(\alpha,\beta))=wt(\alpha)+wt(\beta)\end{equation} 
\end{proposition}
\begin{proof} Consider the boundary strata of the moduli space $\bar{\mathscr{R}}^3_{(0,1)}(x_0;x_1,x_2)$, where we have three boundary marked points and two interior marked points. We can exclude sphere bubbles through a correct choice of bounding cycle. There are still 6 kinds of degeneration (see Figure~\ref{fig:product}) of this moduli space. Three degenerations shown in the first row compute \begin{equation}\phi(\mu^2(x_1,x_2))\end{equation}\begin{equation}\mu^2(x_1,\phi(x_2))\end{equation}\begin{equation}\mu^2(\phi(x_1),x_2)\end{equation} The three degenerations shown in the second row compute \begin{equation}\mu^1\phi^2(x_1,x_2)\end{equation}\begin{equation}\phi^2(\mu^1(x_1),x_2)\end{equation}\begin{equation}\phi^2(x_1,\mu^1(x_2))\end{equation} If we pass to homology, those terms with $\mu^1$ will vanish and thus we have the following relation by counting all the boundary components of $1$ dimensional moduli space $\bar{\mathscr{R}}^2_{(0,1)}(x_0;x_1,x_2)$
\begin{equation}\phi(\mu^2(x_1,x_2))=\mu^2(x_1,\phi(x_2))+\mu^2(\phi(x_1),x_2)\end{equation} This is equivalent to \begin{equation}wt(\mu^2(x_1,x_2))\mu^2(x_1,x_2)=\mu^2(x_1,wt(x_2)x_2)+\mu^2(wt(x_1)x_1,x_2)=(wt(x_1)+wt(x_2))\mu^2(x_1,x_2).\end{equation}, which proves the result. 
\end{proof}
\begin{figure}
\includegraphics[width=6cm]{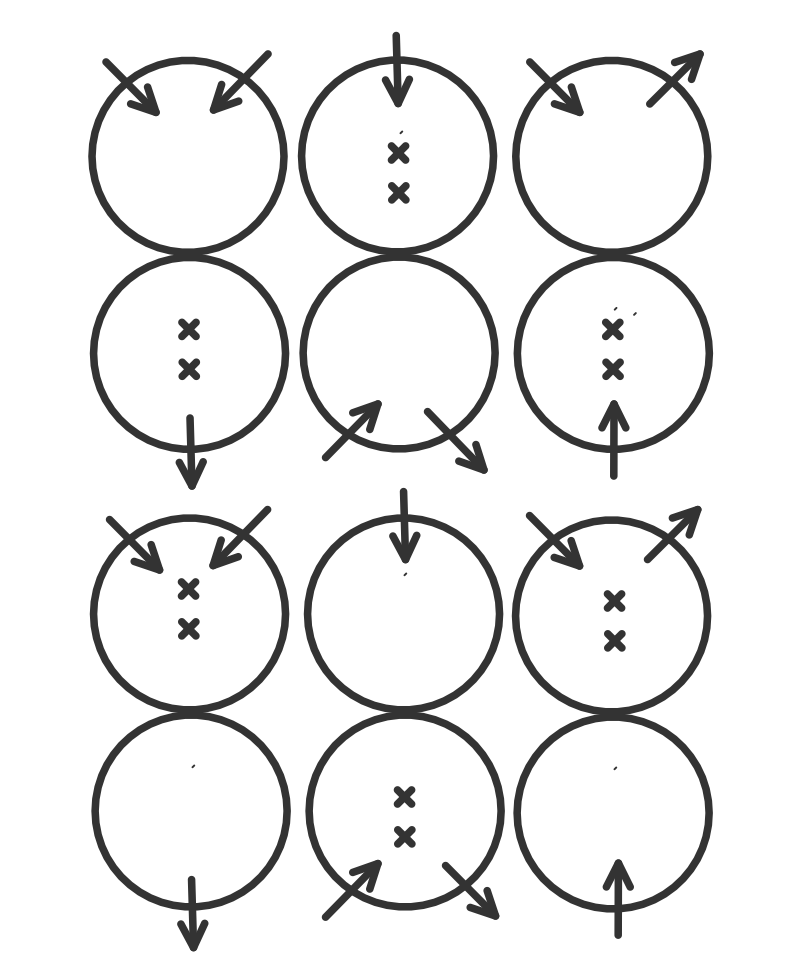}
\centering
\caption[6 possible breaks in the boundary]{6 possible breaks in the boundary. The incoming arrows are inputs and the outgoing ones are outputs}
\label{fig:product}
\end{figure}

\begin{remark}\label{rmk:product} By studying the similar set up with more boundary marked points, we can generalize the result above to higher Floer products. Specifically, when $n=3$, we have $wt(\mu^3(x_1,x_2,x_3))=wt(x_1)+wt(x_2)+wt(x_3)$. 
\end{remark}

\subsection{A long exact sequence of $Kh^{*,*}_{symp}(L)$}\label{ch4}
Abouzaid and Smith constructed a long exact sequence from an exact triangle of bimodules over the Fukaya category of $\mathscr{Y}_n$, see \cite[Equation 7.9]{AS19}:
\begin{equation}\ldots\to Kh_{symp}^{*}(L_+)\to Kh_{symp}^{*}(L_0)\to Kh_{symp}^{*+2}(L_\infty)\to \ldots
\end{equation}
where $L_+$ is a link diagram with a positive crossing and $L_0$ and $L_\infty$ are diagrams given by $0$ or $\infty$ resolutions at the positive crossing. The goal of this chapter is to give an explicit construction of such a long exact sequence with the framework of bridge diagrams that preserves the weight grading, just like the combinatorial Khovanov homology. 


We use the following local diagrams for the computation: we name the blue curves $\beta$, green curves $\gamma$ and yellow curves $\delta$, respectively, in Figure~\ref{fig:exact}, and they share all the other components (and to make it easier to discuss in the future, we move $\gamma$ and $\delta$ slightly away from $\beta$ in all the other components so that they don't intersect in the interior of the arcs). Pairing $\alpha$ with $\beta$ gives $L_+$, $\alpha$ with $\gamma$ gives $L_0$ and $\alpha$ with $\delta$ gives $L_-$. We have the following exact sequence:
\begin{proposition}\cite[Proposition 7.4]{AS19} If we have $\alpha$, $\beta$, $\gamma$ and $\delta$ curves presented like Figure~\ref{fig:exact} locally and $\beta$, $\gamma$ and $\delta$ are the same apart from this region, then we have the following exact sequence
\begin{equation}\label{eqt:exact2}\ldots\xrightarrow{c_1}  HF^{*}(\mathcal{K}_\alpha,\mathcal{K}_\beta)\xrightarrow{c_2} HF^{*}(\mathcal{K}_\alpha,\mathcal{K}_\gamma)\xrightarrow{c_3} HF^{*+2}(\mathcal{K}_\alpha,\mathcal{K}_\delta)\xrightarrow{c_1} \ldots
\end{equation}

\begin{figure} 
\includegraphics[width=6cm]{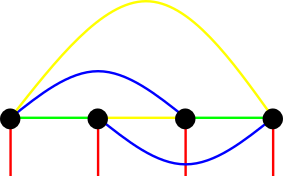}
\centering
\caption[Exact triangle of Lagrangians]{Exact triangle of Lagrangians. The red, blue, green and yellow are $\alpha$, $\beta$, $\gamma$ and $\delta$ curves}
\label{fig:exact}
\end{figure}

In particular, there are elements $c_1 \in CF^{*}(\mathcal{K}_\beta,\mathcal{K}_\delta)$, $c_2 \in CF^{*}(\mathcal{K}_\gamma,\mathcal{K}_\beta)$ and $c_3 \in CF^{*}(\mathcal{K}_\delta,\mathcal{K}_\gamma)$ such that the maps above are Floer products with the corresponding elements. 
\end{proposition}
\begin{proof} Abouzaid and Smith proved that there is an exact triangle of bimodules of the Fukaya category of $\mathscr{Y}_n$ among the identity bimodule, the cup-cap bimodule and the bimodule representing an half-twist $\tau$. Evaluating these bimodules at $\mathcal{K}_\gamma$ as the second object, we have an exact triangle of one-sided modules between $\mathcal{K}_\gamma$, $\mathcal{K}_\delta$ and an one-sided module that is equivalent to $HF^*(\bullet,\mathcal{K}_\beta)$, such that the maps connecting those terms are Floer products with some elements $c_1 \in CF^{*}(\mathcal{K}_\beta,\mathcal{K}_\delta)$, $c_2 \in CF^{*}(\mathcal{K}_\gamma,\mathcal{K}_\beta)$ and $c_3 \in CF^{*}(\mathcal{K}_\delta,\mathcal{K}_\gamma)$. Evaluating these one-sided modules at $\mathcal{K}_\alpha$, we have the desired long exact sequence. 
\end{proof}

Now we need to show that this long exact sequence preserves the relative weight grading in the sense that each element summing to $c_i$ has the same weight grading so that $c_i$ has a well-defined weight grading, and moreover the weight gradings of $c_1$, $c_2$ and $c_3$ sum to $0$. With a closer look into the diagram, we have the following observation:

\begin{lemma} Each pair of $\beta$, $\gamma$ and $\delta$ forms a bridge diagram for an unlink of $(n-1)$-components without any crossing. 
\end{lemma} 
\begin{proof} In the region shown in Figure~\ref{fig:exact}, each two of $\beta$, $\gamma$ and $\delta$ form a crossingless unknot. In the other regions not shown in the figure, each pair of arcs forms a crossingless unknot component as well. Thus we have $1$ crossingless unknot component in Figure~\ref{fig:exact} and crossingless $(n-2)$ unknot components outside that figure. Together, each two of $\beta$, $\gamma$ and $\delta$ form a bridge diagram for an unlink of $(n-1)$ components without any crossing. 
\end{proof}

\begin{remark} Each of the Floer groups is computed with the second Lagrangian perturbed so that it intersects transversely with the first one. In this subsection, $\mathcal{K}_\beta$, $\mathcal{K}_\gamma$, $\mathcal{K}_\delta$ are all perturbed so they intersect transversely with $\mathcal{K}_\alpha$. We also know that $\beta$, $\gamma$, and $\delta$ only intersect at endpoints, so their corresponding Lagrangians have pairwise transverse intersections. We can now assume the perturbations we apply to $\mathcal{K}_\beta$, $\mathcal{K}_\gamma$, and $\mathcal{K}_\delta$ are small enough that they still pairwise intersect transversely, while they still intersect transversely with $\mathcal{K}_\alpha$. We will keep abusing the notation of Lagrangians $\mathcal{K}_\beta$, $\mathcal{K}_\gamma$, $\mathcal{K}_\delta$ and the Lagrangians after perturbations when pairing with $\mathcal{K}_\alpha$. 
\end{remark}

From the computations of \cite{AS19}, we have
\begin{corollary}\label{cor:unlink1} With some grading shifts 
\begin{equation}CF^*(\mathcal{K}_\beta, \mathcal{K}_\gamma)\cong CF^*(\mathcal{K}_\gamma, \mathcal{K}_\delta)\cong CF^*(\mathcal{K}_\delta, \mathcal{K}_\beta)\cong \bigotimes^{n-1} H^*(S^2)\end{equation} and thus all generators are cocycles. 
\end{corollary}

\begin{proposition}\label{prop:es} Fix a choice of equivariant structures on Lagrangians, there is a well-defined weight grading for $c_1$, $c_2$ and $c_3$.  
\end{proposition}
\begin{proof}From Corollary~\ref{cor:unlink1}, we know each of the Floer cochains and cohomologies above can be made such that weight grading equals homological grading. Together with the observation that choices of equivariant structures will only apply overall grading shifts to all weight gradings, thus we know the weight grading of any element in $CF^n$ must be the same. If we fix a set of equivariant structures on $\mathcal{K}_\beta$, $\mathcal{K}_\gamma$ and $\mathcal{K}_\delta$, we have a well-defined weight grading for each $c_i$ from its homological grading plus the effect a grading shift from changing the equivariant structure from the standard one. 
\end{proof}

Before we prove that $wt(c_1)+wt(c_2)+wt(c_3)=0$. Recall the following lemma from Seidel:

\begin{lemma}\cite[Lemma 3.7]{Sei08} A triple $\mathcal{K}_\beta$, $\mathcal{K}_\gamma$ and $\mathcal{K}_\delta$ form an exact triangle of Lagrangians if and only if there exist $c_1 \in CF^{1}(\mathcal{K}_\beta,\mathcal{K}_\delta)$, $c_2 \in CF^{0}(\mathcal{K}_\gamma,\mathcal{K}_\beta)$, $c_3 \in CF^{0}(\mathcal{K}_\delta,\mathcal{K}_\gamma)$, $h_1\in HF^0(\mathcal{K}_\gamma,\mathcal{K}_\beta)$, $h_2\in HF^0(\mathcal{K}_\beta,\mathcal{K}_\delta)$ and $k\in HF^{-1}(\mathcal{K}_\beta,\mathcal{K}_\beta)$ such that
\begin{equation}\label{eqt:6.6} \mu^1(h_1)=\mu^2(c_3,c_2)
\end{equation}
\begin{equation}\label{eqt:6.7} \mu^1(h_2)=-\mu^2(c_1,c_3)
\end{equation}
\begin{equation} \mu^1(k)=-\mu^2(c_1,h_1)+\mu^2(h_2,c_2)+\mu^3(c_1,c_3,c_2)-e_{\mathcal{K}_\beta}
\end{equation}
\end{lemma}

\begin{lemma} For any choice of equivariant structures on Lagrangians, the sum of weight gradings $wt(c_1)+wt(c_2)+wt(c_3)=0$. 
\end{lemma}
\begin{proof} We know $\mu^1$ vanishes on all the Floer groups above. If we know that $h_1=0$ and $h_2=0$, then the equation $(5.6)$ becomes $\mu^3(c_1,c_3,c_2)=e_{\mathcal{K}_\beta}$. Together with the fact weight grading is compatible with Floer product, see Remark~\ref{rmk:product}, we know that $wt(c_1)+wt(c_2)+wt(c_3)=wt(e_{\mathcal{K}_\beta})$. From \cite[Lemma 4.10]{AS16}, no matter which equivariant structure we choose on $\mathcal{K}_\beta$, the identity always has the weight grading $0$. 
To see $h_1=0$ and $h_2=0$, we need to look into the absolute grading. We claim that $HF^*(\mathcal{K}_\gamma,\mathcal{K}_\beta)$ and $HF^*(\mathcal{K}_\beta,\mathcal{K}_\delta)$ are supported in odd degrees. This is because pairing $\gamma$ with $\beta$ gives a flattened braid diagram for braid $\sigma$ with writhe $1$. The number of strands is even and thus this group is supported in odd degrees. Similarly, pairing $\beta$ with $\delta$ gives $\sigma^{-1}$, this also gives an odd degree supported group. Notice that the third map $c_3$ should have degree $2$ in the absolute grading case, thus the exact triangle is actually between $\mathcal{K}_\beta$, $\mathcal{K}_\gamma$ and $\mathcal{K}_\delta[2]$. But shifting the grading of Floer groups by $2$ does not change the fact that $h_1$ and $h_2$ have even degrees. Thus, they must be $0$. 
\end{proof}

Thus we conclude the following, 
\begin{proposition} The long exact sequence Equation~\ref{eqt:exact2}
\begin{equation}\ldots\xrightarrow{c_1}  HF^{*,wt_1}(\mathcal{K}_\alpha,\mathcal{K}_\beta)\xrightarrow{c_2} HF^{*,wt_2}(\mathcal{K}_\alpha,\mathcal{K}_\gamma)\xrightarrow{c_3} HF^{*+2,wt_3}(\mathcal{K}_\alpha,\mathcal{K}_\delta)\xrightarrow{c_1} \ldots
\end{equation}
decomposes with respect to weight gradings. 
\end{proposition}
\begin{proof} From Proposition~\ref{prop:es}, we know that there is a well-defined weight grading for $c_1$, $c_2$ and $c_3$. If we start with the first group $ HF^{*,wt_1}(\mathcal{K}_\alpha,\mathcal{K}_\beta)$, the only non-trivial map will be at weight grading $wt_1+wt(c_2)$ because the weight grading is compatible with Floer products. The same goes for $wt_3=wt_1+wt(c_3)+wt(c_2)$. The next weight grading should be $wt_1+wt(c_2)+wt(c_3)+wt(c_1)=wt_1$, which is exactly where we started with the first group. 
\end{proof}

\section{Proof of the main theorem via a bigraded isomorphism}\label{ch5}

Now we have enough to prove our main theorems, Theorem~\ref{thm:main}
\begin{theorem}\label{thm:bigrading} Symplectic Khovanov cohomology, graded by $(gr,wt)$, and Khovanov homology, graded by $(i,j)$, are isomorphic as bigraded vector spaces over any characteristic zero field, where the gradings are related by $gr=i-j$ and $wt=-j+c$, with an ambiguity of a grading shift $c$ on relative weight grading. \end{theorem}

Our isomorphism is a graded refinement of the isomorphism in \cite{AS19}. 
\begin{proposition}\cite[Theorem 7.5]{AS19} For any bridge diagram $L$, we have an isomorphism $H$ such that
\begin{equation} H: Kh^*_{symp}(L)\to Kh^*(L)
\end{equation}
\end{proposition}

In \cite{AS19}, we can only conclude from the original argument of Abouzaid and Smith that $H$ is canonical for braid closures. With Proposition~\ref{prop:rela} that symplectic Khovanov cohomology is defined canonically (and the same for combinatorial Khovanov homology), we can claim that $H$ is canonical for any link diagram. This is an isomorphism with only information on the homological grading. But the result of Abouzaid and Smith implies that the long exact sequence in Equation~\ref{eqt:exact2} commutes with the corresponding long exact sequence in Khovanov homology.

\begin{lemma} Fix a link diagram $L_+$ and its unoriented resolutions $L_0$ and $L_\infty$ at one of the crossings. We represent their corresponding bridge diagrams with $(\alpha,\beta)$, $(\alpha,\gamma)$ and $(\alpha,\delta)$-arcs respectively such that $\alpha$, $\beta$, $\gamma$ and $\delta$ are locally shown in Figure~\ref{fig:exact}. The isomorphism $H$ is compatible with the exact sequence Equation~\ref{eqt:exact2}, i.e. the following diagram commutes
\[\begin{tikzcd}
\hspace{-1cm}HF^{*}(\mathcal{K}_\alpha,\mathcal{K}_\gamma) \arrow{r} \arrow[swap]{d}{H} & HF^{*+2}(\mathcal{K}_\alpha,\mathcal{K}_\delta) \arrow{r} \arrow[swap]{d}{H}& HF^{*+1}(\mathcal{K}_\alpha,\mathcal{K}_\beta) \arrow{r} \arrow[swap]{d}{H}& HF^{*+1}(\mathcal{K}_\alpha,\mathcal{K}_\gamma) \arrow{r} \arrow[swap]{d}{H}&HF^{*+3}(\mathcal{K}_\alpha,\mathcal{K}_\delta) \arrow{d}{H} \\
\hspace{-1cm}Kh^*(L_0) \arrow{r} & Kh^{*+2}(L_\infty) \arrow{r} &Kh^{*+1}(L_+) \arrow{r} &Kh^{*+1}(L_0) \arrow{r} &Kh^{*+3}(L_\infty) 
\end{tikzcd}
\]

where the first line is the exact sequence of Equation~\ref{eqt:exact2} and the second line is the exact sequence for combinatorial Khovanov homology with grading $i-j$.
\end{lemma}
\begin{proof} In the proof of the Abouzaid-Smith isomorphism, they show that the cup functors of Khovanov and symplectic Khovanov are identified with the isomorphism in the arc algebra. Cap functors in both cases are adjoint to the corresponding cup functors. The horizontal maps in the first squares are given by applying the same cap-cup functor in each case, and thus they commute with the isomorphisms in the first square. 

The diagram naturally commutes if we replace the third group of each row with the mapping cone of the other two. The third group of each row is isomorphic to a mapping cone of the other two, see \cite[Proposition 7.4]{AS19}. Moreover, in the proof of \cite[Theorem 7.6]{AS19}, the isomorphism $H$ and the long exact sequences are constructed via the mapping cones and thus factor through the cones of horizontal maps of the first square. As a result, the second and third squares are also commutative.
\end{proof}

\begin{corollary} The diagram involving the exact sequences of resolving a negative crossing is also commutative. 
\end{corollary}

We compared the bigradings of two theories at the end of Section~\ref{3.4} for unlinks:

\begin{proposition}\cite[Proposition 6.11]{AS16} \label{prop:unlink} The isomorphism $H$ preserves the weight grading if $D$ is crossingless diagram, with grading correspondence $gr=i-j$ and $wt=-j$. 
\end{proposition}

\begin{proof}[Proof of Theorem~\ref{thm:main}] We only need to prove that for any fixed Jones grading $j_0$, $H$ is also an isomorphism with $wt=-j_0+c$ with some grading shift $c$, 
\begin{equation} H: Kh^{*,-j_0}_{symp}(L)\to Kh^{*,j_0}(L)
\end{equation}
We prove that this statement is true for any link (bridge) diagram by induction on the number of crossings. The base case for unlinks is proved with Proposition~\ref{prop:unlink}. 

Now we assume that $L_+$ is a link diagram with n crossings, whereas its resolutions $L_0$ and $L_\infty$ have $(n-1)$ crossings. Let us assume we are doing resolutions at a positive crossing, if all crossings are negative, a similar argument can be applied for the exact sequences induced by doing resolution at a negative crossing. From the inductive assumption, $H$ are bigraded isomorphisms for $L_0$ and $L_\infty$. Let us fix a Jones grading $j_1$ on $Kh(L_0)$. The maps in the long exact sequence will be trivial unless the Jones grading $j_2$ on $Kh(L_\infty)$ is $j_1-3v-2$, where $v$ is the signed count of crossings between this arc and other components, and $j_3$ on $Kh(L_+)$ to be $j_1+1$. Thus we can decompose our commutative diagram with respect to the Jones grading: 

\small\[\begin{tikzcd}
\hspace{-2.5cm}HF^{*,-j_1}(\mathcal{K}_\alpha,\mathcal{K}_\gamma) \arrow{r}{c_2} \arrow[swap]{d}{H} & HF^{*+2,-j_2}(\mathcal{K}_\alpha,\mathcal{K}_\delta) \arrow{r}{c_3} \arrow[swap]{d}{H}& HF^{*+1,j'_3}(\mathcal{K}_\alpha,\mathcal{K}_\beta) \arrow{r}{c_1} \arrow[swap]{d}{H}& HF^{*+1,-j_1}(\mathcal{K}_\alpha,\mathcal{K}_\gamma) \arrow{r}{c_2} \arrow[swap]{d}{H}&HF^{*+3,-j_2}(\mathcal{K}_\alpha,\mathcal{K}_\delta) \arrow{d}{H} \\
\hspace{-1.5cm}Kh^{*,j_1}(L_0) \arrow{r} & Kh^{*+2,j_2}(L_\infty) \arrow{r} &Kh^{*+1,j_3}(L_+) \arrow{r} &Kh^{*+1,j_1}(L_0) \arrow{r} &Kh^{*+3,j_2}(L_\infty) 
\end{tikzcd}
\]
\normalsize where the weight grading $j'_3$ is given by $-j_2+wt(c_3)$. The first, second, fourth, and fifth columns are all isomorphisms, so by the five lemma, we conclude that the third column is also an isomorphism and thus we know $H$ is also a bigraded isomorphism for $L_+$. As for the grading correspondence, because the $c_i$ all have fixed weight grading after specifying the choice of equivariant structures, if we change $j_1$ by any number $k$, we change $j_3$ also by $k$. As for the first row, if $-j_1$ is changed into $-j_1-k$, this will result in $j'_3$ shifting by $-k$ as well. This is enough to show that as a relative grading, the weight grading recovers the Jones grading as a relative grading. 


\end{proof}

As a corollary of the theorem above, we can also conclude Theorem~\ref{thm:relative} that the relative weight grading is independent of the choice of link diagrams.
\begin{proof}[Proof of Theorem~\ref{thm:relative}] For any two bridge diagrams $D$ and $D'$ representing a link $L$, relative weight grading $wt$ and $wt'$ can be defined on $D$, and respectively $D'$. Theorem~\ref{thm:main} indicates that both $wt$ and $wt'$ coincide with $-j$ with as relative gradings. Thus $wt$ and $wt'$ are the same as relative gradings.  
\end{proof}

At the writing of this paper, the author can not provide a pure symplectic proof of Theorem~\ref{thm:relative} without referring to the invariance of the bigrading on Khovanov homology. Such a proof consists of the invariance of weight grading under isotopy, handeslide and stabilization. 

Isotopy and handleslide invariance could be potentially proved via Proposition~\ref{prop:product}, if we look close enough into the bridge diagrams and the Floer products. Isotopy and handleslide invariance can also be deduced from Hamiltonian isotopy invariance. But general Hamiltonian isotopy invariance will require virtual perturbations because some of the transversality assumptions will not be preserved under general isotopies, say, $L$ might bound some Maslov zero discs after isotopy.  

Stabilization invariance is harder than the other two. It requires some degeneration arguments in the Hilbert scheme setup, relating holomorphic discs in $\mathscr{Y}_n$ and discs in $\mathscr{Y}_{n-1}$ so that differentials in the stabilized diagram can be identified with differentials in the original diagram. However, degeneration arguments are only used in the original nilpotent slice definition of Seidel and Smith, while the weight grading is defined using Hilbert schemes. 

\bibliographystyle{unsrt}
\bibliography{references}

\end{document}